\documentclass[a4paper,12pt]{amsart}
\usepackage{amssymb,amscd,amsthm,mathrsfs}
\usepackage{epsfig,graphicx,color,
mathrsfs}  %figuras
\usepackage[ansinew]{inputenc}
\usepackage[centertags]{amsmath}
\usepackage[colorlinks=true]{hyperref}
\hypersetup{urlcolor=blue, citecolor=red}

 \def\RR{{\mathbb R}}  
   
 \def\ZZ{{\mathbb Z}}

\def\cA{\mathcal{A}}  \def\cG{\mathcal{G}}  \def\cS{\mathcal{S}}
    
\def\cC{\mathcal{C}}   \def\cO{\mathcal{O}} \def\cU{\mathcal{U}}
\def\cD{\mathcal{D}}    \def\cV{\mathcal{V}}
\def\cE{\mathcal{E}}    
   \def\cR{\mathcal{R}}

\newtheorem*{teo*}{Main Theorem}
\newtheorem*{teo1*}{Main Theorem (first version)}
\newtheorem*{criterion}{Isolated Point Criterion}
\newtheorem*{af}{Claim}

\newtheorem{teo}{Theorem}
\newtheorem{lema}[teo]{Lemma}
\newtheorem{prop}[teo]{Proposition}
\newtheorem{cor}[teo]{Corollary}

\newtheorem{quest}{Question}

\theoremstyle{definition}
\newtheorem{defi}{Definition}

\theoremstyle{remark}
\newtheorem{obs}{Remark}

\newcommand{\eps}{\varepsilon}

\newcommand{\en}{\subset}

\DeclareMathOperator{\Diff}{Diff}

\DeclareMathOperator{\interior}{Int}

\title{Tame dynamics and robust transitivity}

\author[C. Bonatti]{C. Bonatti}
\address{
CNRS - IMB. UMR 5584.
\rm Universit\'e de Bourgogne, 21004 Dijon, France}
\email{bonatti "at" u-bourgogne.fr}

\author[S. Crovisier]{S. Crovisier}
\address{CNRS - LMO. UMR 8628.
\rm Universit\'e Paris-Sud 11, 91405 Orsay, France}
\email{sylvain.crovisier "at" math.u-psud.fr}

\author[N. Gourmelon]{N. Gourmelon}
\address{IMB. UMR 5251. \rm Universit\'e Bordeaux 1, 33405 Talence, France}
\email{Nicolas.Gourmelon "at" math.u-bordeaux1.fr}

\author[R. Potrie]{R. Potrie}
\address{CMAT.
\rm Facultad de Ciencias, Universidad de la Rep\'ublica, Uruguay}
\email{rpotrie "at" cmat.edu.uy}

\date{\today}

\thanks{Partially supported by the ANR project \emph{DynNonHyp} BLAN08-2 313375.}

\begin{document}
\maketitle

\begin{abstract}
One main task of smooth dynamical systems consists in finding a good decomposition into elementary pieces
of the dynamics. This paper contributes to the study of chain-recurrence classes.
It is known that $C^1$-generically, each chain-recurrence class containing a periodic orbit
is equal to the homoclinic class of this orbit. Our result implies that in general this property is fragile. 

We build a $C^1$-open set $\cU$ of tame diffeomorphisms
(their dynamics only splits into finitely many chain-recurrence classes)
such that for any diffeomorphism in a $C^\infty$-dense subset of $\cU$,
one of the chain-recurrence classes is not transitive (and has an isolated point).
Moreover, these dynamics are obtained among partially hyperbolic systems with one-dimensional
center.

%\bigskip

%\noindent%{\bf Keywords:}
%\medskip

%\noindent {\bf MSC 2000:} 37C05, 37C20, 37C25, 37C29, 37D30.
\end{abstract}

\section{Introduction}
In the setting of hyperbolic diffeomorphisms, Smale's spectral decomposition theorem organizes the global dynamics by decomposing it into a finite number of pieces. After the first examples of open sets of non hyperbolic systems \cite{AS,New} people focused on non hyperbolic dynamics trying to recover a decomposition in pieces in this new setting. Many such examples have been obtained, and we can distinguish two main different behaviours:
\begin{itemize}
\item[--] The dynamics can be globally ``undecomposable": there are open sets of \emph{transitive} systems, i.e.
having a half-orbit which is dense in the whole manifold, see~\cite{Sh,Ma,persistence}. In the same spirit, other examples present a decomposition in finitely many disjoint, isolated, robustly transitive compact invariant sets \cite{Ca,BV,persistence}. 
\item[--] On the opposite, the dynamics of generic systems in some $C^r$-open sets ($r\geq 1$) can split into
infinitely many pieces, see~\cite{New,universal}.
Necessarily in this case, some pieces of the decomposition
should be accumulated by other ones.
\end{itemize}
These two patterns (finite or infinite number of indecomposable pieces) are called  \emph{tame} and \emph{wild} systems: in order to formalize these notions we have to explain what we mean by ``pieces''. Several  definitions  have been proposed. 

%An idea for exhibiting examples of non-hyperbolic isolated pieces consists in the construction of
%\emph{robustly transitive sets} as defined in~\cite{BDP}: one considers two open sets $V\subset U\subset M$
%and a $C^1$ open set of systems $f$
%satisfying $f(\overline U)\subset U$,$f(\overline V)\subset V$ and having a dense orbit in the maximal
%invariant set of $U\setminus V$.
%The first examples such that $U\neq M$ were given by~\cite{Ca,BV}, and possess
%a non-hyperbolic robustly transitive attractor.

\begin{itemize}
\item[--] A natural way for a piece to be  dynamically indecomposable is to be transitive. 
Among notable transitive sets are the {\em homoclinic classes} of hyperbolic periodic orbits, that is, the closure of the transversal intersection of ther invariant manifolds. These classes contain dense subsets of hyperbolic periodic orbits that are {\em homoclinically related}: the stable manifold of each of these orbits intersects transversally the unstable manifold of each other.  

\item[--] Other natural candidates for pieces are the \emph{maximal transitive sets}  \cite{universal}. They always exists but may fail to be disjoint and  to capture a large part of the dynamics. 

\item[--]    By weakening the notion of transitivity, Conley  \cite{Co} defined the  very general notion of \emph{chain-recurrence classes}  for homeomorphisms
on a compact manifold $M$. The \emph{chain-recurrent set} $\cR(h)$
is the set of points $x$ contained in closed $\varepsilon$-pseudo-orbits for every $\varepsilon>0$.  This set splits into invariant compact subsets called
\emph{chain-recurrence classes}: two points $x,y\in \cR(f)$ are in a same class if for every $\varepsilon>0$ they are contained in a same closed $\varepsilon$-pseudo-orbit. The chain-recurrence classes cover the whole interesting dynamics, as this notion of recurrence is the weakest possible: the chain-recurrence set contains the limit sets and the non-wandering set.
Moreover, the different chain-recurrence classes
can be separated by filtrations of the dynamics.  
\end{itemize}

Each of these notions has its advantages and drawbacks and in general they do not coincide.
When $f$ is a $C^1$-generic diffeomorphism, one gets better properties.
For instance, \cite{BC} showed that any chain-recurrence class
$\cC$ containing a hyperbolic periodic orbit $O$ coincides with the homoclinic class $H(O)$ of $O$,
and is the unique maximal transitive set containing $O$.
Of course, these properties are not satisfied by all diffeomorphisms, but once they hold $C^1$-generically, one can expect that they persist on larger class of systems, in particular for some close $C^r$-diffeomorphisms. More precisely one can ask:

\begin{quest}
Consider a hyperbolic periodic orbit $O_f$ whose continuation exists on an open set $\cU_0$
of $C^1$-diffeomorphisms. Under what conditions
does there exists a dense open subset $\cU\subset \cU_0$ such that
the chain-recurrence class containing $O_f$ is transitive for any $f\in \cU$?
coincides with the homoclinic class $H(O_f)$?
\end{quest}

This question could be asked for an important particular case.
A chain-recurrence class $\cC$ of a diffeomorphisms $f$ is \emph{robustly isolated} (for the $C^1$-topology) if there exists a neighborhood $\cO$ of $\cC$ and a $C^1$-neighborhood $\cU$ of $f$
such that, for every $g\in\cU$, the set $\cR(g)\cap \cO$ of chain-recurrent points in $\cO$
coincides precisely with a chain-recurrence class $\cC_g$ of $g$ called its \emph{continuation}.

Following~\cite{Ab, BDcycl}, one says that a diffeomorphism $f$ is \emph{tame} if each of its chain-recurrence classes is robustly isolated: in this case, the number of chain-recurrence classes
is finite and constant on a $C^1$-neighborhood of $f$.
Tame diffeomorphisms are the simplest ones and the dynamics of $C^1$-generic
such systems is easier to describe: for instance, all its chain-recurrence classes are homoclinic classes~\cite{BC}.
Until now, the unique known examples of tame dynamics have robustly transitive chain-recurrence classes, i.e.
robustly isolated classes whose continuation remains transitive for any $C^1$-perturbation.
For this reason it was natural to ask:

\begin{quest}[\cite{BC}, problem 1.3]\label{q.1} Is there a dense subset $\cD$ in the space $\Diff^{1}(M)$ of diffeomorphisms
such that every robustly isolated chain-recurrence class of any $f\in \cD$ is robustly transitive? 
\end{quest}
The present paper gives a negative answer to this question.
We denote there by $\Diff^r(M)$ the space of $C^r$-diffeomorphisms of $M$.

\begin{teo1*} Any compact manifold $M$, with $\dim(M)\geq 3$, admits a smooth diffeomorphism $f$ with a robustly isolated chain-recurrence class $\cC_f$ satisfying for any $r>1$ the following property.

The diffeomorphisms $g$  such that the continuation $\cC_g$ is not transitive
form a $C^r$-dense subset of a $C^1$-neighborhood of $f$ in $\Diff^r(M)$.
\end{teo1*}

Actually, the fact for a class to be robustly isolated brings several restriction on the dynamics.
On surfaces, it is well-known that such a class needs to be far from homoclinic tangencies~\cite{New,PS}.
Also it was shown in~\cite{BDP} that it presents dominated splittings and volume hyperbolicity.
Indeed the dominated splitting is the structure which is antagonist
to the homoclinic tangencies~\cite{W,G}.
Conversely, we expect (see~\cite{bonatti-conjecture}, conjecture 11) that, $C^1$-generically,
any chain-recurrent classes exhibiting enough hyperbolicity is robustly isolated:
this should be the case for \emph{partially hyperbolic} classes with a one-dimensional center bundle
(see a precise definition in section~\ref{ss.partial}).
\medskip

The aim of this paper is to build an example of a robustly isolated class as close as possible to being hyperbolic,
in particular $C^1$-far from homoclinic tangencies, which nevertheless is not robustly transitive.
There is the precise statement of our result.

\begin{teo*} When $\dim(M)\geq 3$, there exist a $C^1$-open set $\cU \en \Diff^r(M)$, $1\leq r \leq \infty$,
a $C^r$-dense subset $\cD\subset \cU$ and an open set $U\en M$ with the following properties:
\begin{itemize}
\item[(I)] \emph{Isolation}: For every $f\in \cU$, the set $\cC_f:=U\cap \cR(f)$
is a chain-recurrence class.
\item[(II)] \emph{Non-robust transitivity}: For every $f\in \cD$, the class $\cC_f$
is not transitive.
\item[(III)] \emph{Partial hyperbolicity}: For any $f\in \cU$, the chain-recurrence class $\cC_f$ is partially hyperbolic with a one-dimensional central bundle.

\end{itemize}
More precisely:
\begin{itemize}
\item[(1)] For any $f\in \cU$ there exists a subset $H_f \subset \cC_f$ that
coincides with the homoclinic class of any hyperbolic periodic $x\in \cC_f$.
Moreover, each pair of hyperbolic periodic points in $\cC_f$ with the same stable dimension is homoclinically related.
\item[(2)] For any $f \in \cU$ there exist two hyperbolic periodic points $p,q \in \cC_f$ satisfying $\dim E^s_p=\dim E^s_q+1$ and $\cC_f$ is the disjoint union of $H_f$ with $W^u(p)\cap W^s(q)$. Moreover the points of $W^u(p)\cap W^s(q)$ are isolated in $\cC_f$.

In particular, if $W^u(p)\cap W^s(q)\neq \emptyset$, the class $\cC_f$ is not transitive.
\item[(3)] One has $\cD:= \{ f\in \cU \ : \ W^u(p)\cap W^s(q)\neq \emptyset\}$.
Moreover, this set is a countable union of one-codimensional submanifolds of $\cU$.
\item[(4)] The chain-recurrent set of any $f\in \cU$ is the union of
$\cC_f$ with a finite number of hyperbolic periodic points (which depend continuously on $f$).
\end{itemize}
\end{teo*}

\begin{obs} The isolated points $\cC_f\setminus H_f$ are nonwandering for $f$.
However, they do not belong to $\Omega(f|_{\Omega(f)})$
(since they are isolated in $\Omega(f)$ and non-periodic).
\end{obs}

Let us give now the spirit of our construction. 
We consider two hyperbolic periodic points  $p$ and $q$ with different stable dimension: $\dim(E^s_p)=\dim(E^s_q)+1$. We assume moreover that they robustly belong to the same chain-recurrence class $\cC_f$, which has a partially hyperbolic structure with $1$ dimensional central bundle $E^c$. Such construction can be obtained using blenders, as in~\cite{persistence}.

We can guarantee that a central orientation is preserved so that each point in $\cC_f$
has a right and a left central direction.
We will choose $p$ to be extremal points of the class $\cC_f$ in the following sense (see section~\ref{s.cuspidal}):
we assume that the strong stable manifold $W^{ss}(p)$ of $p$ just meets $\cC_f$ at $p$
and that the intersection $\cC_f\cap W^s(p)$ is contained in a right half submanifold bounded by
$W^{ss}(p)$. In this case the class $\cC_f$ is contained in a cuspidal prism whose edge is the
strong unstable manifold of $p$, see figure~\ref{FiguraCostado}.
One easily verifies that this hypothesis is robust: a fundamental domain of $W^{ss}(p)$
and of the left component of $W^s(p)\setminus W^{ss}(p)$ goes out of a filtrating neighborhood of $\cC_f$.
We call such a point $p$ a \emph{right stable cuspidal point}.
Symmetrically one chooses $q$ to be \emph{left unstable cuspidal}:
$W^{uu}$ and  a right half submanifold bounded by $W^{ss}(p)$ just meet $\cC_f$ at $q$.

As $p$ and $q$ belong to the same class,
small perturbations produce heteroclinic intersections $W^u(p)\cap W^s(q)$.
In particular, any such intersection  belongs to $\cC_f$ but the simultaneous left and right cuspidal geometry implies that it is isolated in $\cC_f$. This formalizes into the following criterion that we use for obtaining
our main theorem.

\begin{criterion}
Let $\cC_f$ be a partially hyperbolic chain-recurrence class. Assume that its central bundle is one-dimensional and endowed with an invariant continuous orientation.
Consider in $\cC_f$ a right stable cuspidal point $p$ and a left unstable cuspidal point $q$,
satisfying $\dim(E^s_p)=\dim(E^s_q)+1$.

Then, any intersection point $x\in W^u(p)\cap W^s(q)$ is isolated in $\cC_f$.
\end{criterion}

Our construction uses strongly the existence of strong stable and strong unstable manifolds
outside of the chain-recurrence class $\cC_f$. Therefore this class cannot be an \emph{attractor},
that is, a transitive set with a neighborhood $U$ satisfying $\cC_f=\cap_{n\geq 0} f^n(U)$;
nor can it be an attractor for $f^{-1}$. The following question remains open.

\begin{quest} Is there a dense $G_\delta$ subset $\cG\subset\Diff^{1}(M)$ such that
every attractor of any $f\in \cG$ is robustly transitive? 
\end{quest}

\section{A mechanism for having isolated points in a chain-recurrence class}\label{sectionMecanisme}

\subsection{Preliminaries on invariant bundles}\label{ss.partial}
Consider $f \in \Diff^1(M)$ preserving a set $\Lambda$.

A $Df$-invariant subbundle $E \en T_\Lambda M$ is {\em uniformly contracted} (resp. {\em uniformly expanded}) if there exists $N>0$ such that for every unit vector $v\in E$,
we have
$$ \|Df^N v \| < \frac 1 2  \qquad (\mbox{resp. } > 2 ).$$
A $Df$-invariant splitting $T_\Lambda M = E^{ss} \oplus E^c \oplus E^{uu}$ is \emph{partially hyperbolic} if $E^{ss}$ is uniformly contracted, $E^{uu}$ is uniformly expanded, both are non trivial, and if there exists $N>0$ such that for any $x\in \Lambda$ and any unit vectors $v_s\in E^{ss}_{x}, v_c \in E^c_{x}$ and $v_u \in E^{uu}_{x}$ we have:
\[  \|Df^N v_s \| < \frac 1 2 \|Df^N v_c \| < \frac 1 4 \|Df^N v_u \|. \]
$E^{ss}$, $E^c$ and $E^{uu}$ are called the {\em strong stable}, {\em center}, and {\em strong unstable} bundles.

\begin{obs}\label{RmqOrientationCentral}
We will sometimes consider a $Df$-invariant continuous orientation of $E^c$.
When $\Lambda$ is the union of two different periodic orbits $O_{p},O_{q}$ and of a heteroclinic orbit
$\{f^n(x)\}\subset W^u(O_{p})\cap W^s(O_{q})$, such an orientation exists if and only if
above each orbit $O_{p}, O_{q}$, the tangent map $Df$ preserves an orientation of the central bundle.

On a one-dimensional bundle, an orientation corresponds to a unit vector field tangent.
\end{obs}

\subsection{Cuspidal periodic points}\label{s.cuspidal}
Let $p$ be a hyperbolic periodic point whose orbit is partially hyperbolic
with a one-dimensional central bundle. When the central space is stable,
there exists a strong stable manifold $W^{ss}(p)$ tangent to $E^{ss}_{p}$ that is invariant by the iterates
$f^\tau$ that fix $p$. It is contained in and separates the stable manifold $W^s(p)$ in two {\em half stable manifolds}
which contain $W^{ss}(p)$ as a boundary.

Let us consider an orientation of $E^c_{p}$. The unit vector defining the orientation goes inward on one half stable manifold of $p$,
that we call the {\em right half stable manifold} $R^s(p)$. The other one is called the {\em left half stable manifold} $L^s(p)$.
These half stable manifolds are invariant by an iterate $f^\tau$
which fixes $p$ if and only if the orientation of $E^c_p$ is preserved by $Df^\tau_p$.
When the central space is unstable, one defines similarly the right and left half unstable manifolds $R^u(p),L^u(p)$.

\begin{defi}
A hyperbolic periodic point $p$ is \emph{right stable cuspidal} if:
\begin{itemize}
\item[--] its orbit is partially hyperbolic, the central bundle is one-dimensional and stable;
\item[--] the left half stable manifold of $p$ intersects the chain-recurrence class of $p$ only at $p$.
\end{itemize}
Symmetrically one defines the \emph{left stable cuspidal points}.
\end{defi}

When the chain-recurrence class $\cC$ containing $p$ is not reduced to the orbit $O_{p}$ of $p$,
this forces the existence of a $Df$-invariant orientation on the central bundle of $O_p$.
In this case, the other half stable manifold intersects $\cC$ at points different from $p$.
The choice of the name has to do with the geometry it imposes on $\cC\cap W^s(p)$ in a neighborhood of $p$, see figure~\ref{FiguraCostado}.
This notion appears in~\cite{fragile}. It is stronger than the notion of \emph{stable boundary points} in \cite{CP}.
We can define in a similar way the \emph{left/right unstable cuspidal points}.

\begin{figure}[ht]\begin{center}
\input{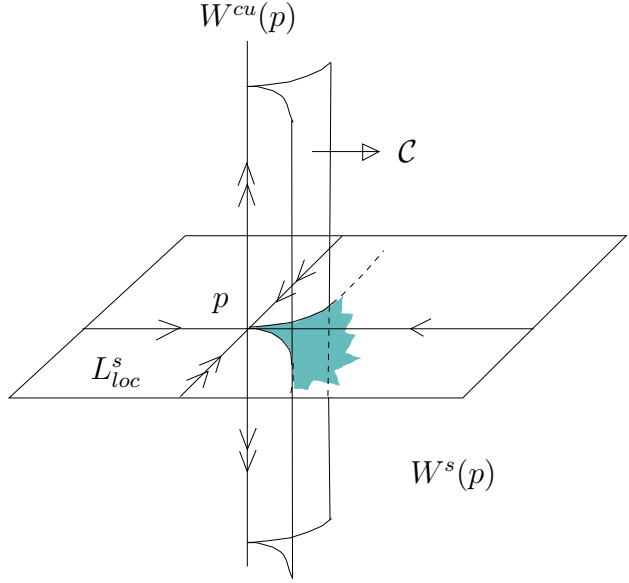}
\caption{\small{Geometry of a chain-recurrence class $\cC$ near a stable cuspidal
fixed point.}} \label{FiguraCostado}
\end{center}\end{figure}

\begin{obs}\label{RmqCuspidalRobust} If $p$ is a stable cuspidal point,
then the hyperbolic continuation $p_g$ is still stable cuspidal for every $g$
that is $C^1$-close to $f$.
Indeed, there exists a compact set $\Delta\subset L^s(p)$ which meets
every orbit of $L^s(p)\setminus \{p\}$ and which is disjoint from $\cR(f)$.
By semi-continuity of the chain-recurrent set, a small neighborhood $V$
of $\Delta$ is disjoint from $\cR(g)$ for any $g$ close to $f$ and meets
every orbit of the continuation of $L^s(p)\setminus \{p\}$.
\end{obs}

\subsection{Description of the mechanism} Let $x$ be a point in a chain-recurrence class $\cC$. We introduce the following assumptions
(see figure \ref{FiguraMecanismo}).
\begin{itemize}
\item[(H1)] $\cC$ contains two periodic points $p,q$ such that $\dim(E^s_{p})=\dim(E^s_{q})+1$.
\item[(H2)] The point $x$ belongs to $W^{u}(p)\cap W^{s}(q)$. The union $\Lambda$ of the orbits of $x,p,q$
has a partially hyperbolic decomposition with a one-dimensional central bundle.
There exists a $Df$-invariant continuous orientation of the central bundle over $\Lambda$
\item[(H3)] For the central orientation on $\Lambda$:
\begin{itemize}\item[(i)] the point $p$ is a right stable cuspidal point;
    \item[(ii)] the point $q$ is a left unstable cuspidal point.\end{itemize}
\end{itemize}
Note that from remark \ref{RmqOrientationCentral} and the fact that a central orientation is preserved for cuspidal points,
a $Df$-invariant continuous orientation of the central bundle over $\Lambda$ always exists.
The following proposition implies the isolation point criterion stated in the introduction.

\begin{figure}[ht]\begin{center}
\input{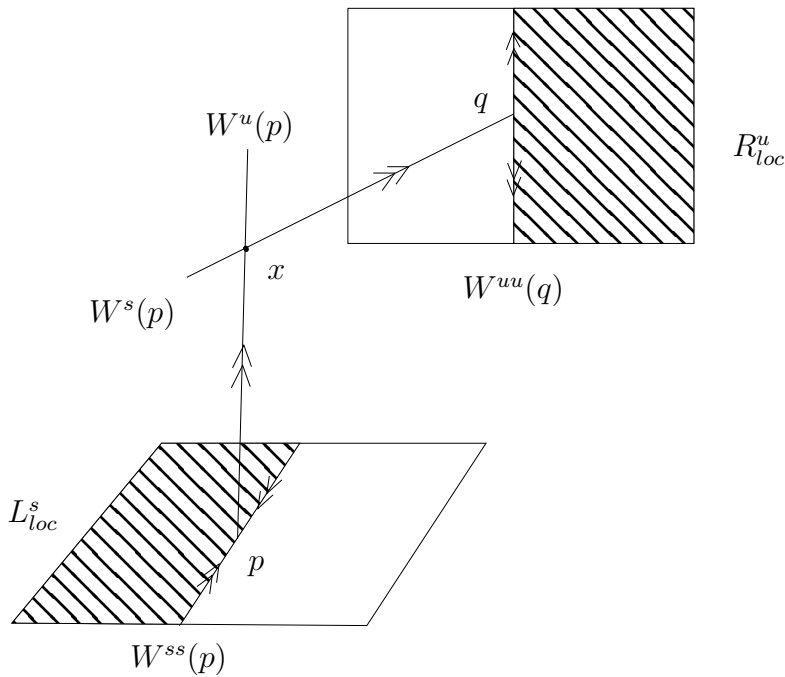}
\caption{\small{Hypothesis (H1)-(H3).}} \label{FiguraMecanismo}
\end{center}\end{figure}

\begin{prop}\label{propMecanisme}
Under (H1)-(H3), the point $x$ is isolated in the chain-recurrence class $\cC$.
In particular, $\cC$ is not transitive and is not a homoclinic class.
\end{prop}

\subsection{Proof of proposition~\ref{propMecanisme}}
Let $q$ be a periodic point whose orbit is partially hyperbolic and
whose central bundle is one-dimensional and unstable.  We shall assume that there is an orientation in $E^c_q$ which is preserved by $Df$.
We fix such an orientation of the central bundle $E^c_q$,
so that the left and right half unstable manifolds of $q$ are defined.
We denote by $d^u+1$ the unstable dimension of $q$.

Any $x \in W^s(q)$ has uniquely defined stable $E^s_x$ and center stable $E^{cs}_x$ directions: the first one is the tangent space $T_xW^s(q)$; a vector
$v\in T_xM\setminus \{0\}$ belongs to the second if the direction of its positive iterates
$Df^{n}(v)$ stays away from the directions of $E^{uu}_q$.
If $E'\subset E$ are two vector subspaces of $T_xM$
such that $E$ is transverse to $E^s_x$ and $E'$ is transverse to $E^{cs}_x$
(hence $E'$ is one-codimensional in $E$), then $F=E^{cs}_x\cap E$ is a one-dimensional space whose forward iterates converge to the unstable bundle over the orbit of $q$. As a consequence, there exists an orientation of $F$ which converges to the orientation of the central bundle by forward iterations.

There is thus a connected component of $E\setminus E'$, such that it intersects $F$ in the orientation of
$F$ which converges towards the central orientation, its closure is the \emph{right half plane} of $E\setminus E'$.
The closure of the other component is the \emph{left half plane} of $E\setminus E'$.

%by forward iterations:
%\begin{itemize}
%\item[--] the sequence $Df^n(E)$ converges towards the unstable space of the orbit of $q$;
%\item[--] for $v\in E\setminus E'$, the direction of $Df^n.v$  converge towards the central bundle.
%\end{itemize}
%The set of $v\in E\setminus E'$, such that the direction of
%$Df^n.v$ converge towards the central orientation, is a component of $E\setminus E'$.
%Its closure is the \emph{right half plane} of $E\setminus E'$.
%The closure of the other component is the \emph{left half plane} of $E\setminus E'$.

Consider a $C^1$-embedding $\varphi\colon [-1,1]^{d^u} \to M$ such that $x:=\varphi(0)$
belongs to $W^s(q)$.

\begin{defi}
The embedding $\varphi$ is {\em coherent with the central orientation at $q$} if
\begin{itemize}
\item[--] $E:=D_0\varphi(\RR^{d^u+1})$ and $E':=D_0\varphi(\{0\}\times \RR^{d^u})$
are transverse to $E^s_z,E^{cs}_z$ respectively;
\item[--] the half-plaque $\varphi([0,1]\times [-1,1]^{d^u})$ is tangent to the right half-plane of $E\setminus E'$.
\end{itemize}
\end{defi}
\medskip

Let $\Delta^u$ be a compact set contained in $R^u(q)\setminus \{q\}$
which meets each orbit of $R^u(q)\setminus \{q\}$.

\begin{lema}\label{LemaLambdaLema} Let $\{\varphi_a\}_{a\in \cA}$ be a continuous family
of $C^1$-embeddings that are coherent with the central orientation at $q$.
Consider some $a_0\in \cA$ and a neighborhood $V^u$ of $\Delta^u$.

Then, there exist $\delta>0$ and some neighborhood $A$ of $a_0$
such that any point $z\in \varphi_a([0,\delta]\times [-\delta,\delta]^{d^u})$
different from $\varphi_a(0)$ has a forward iterate in $V^u$.
\end{lema}
\begin{proof}
Let $\tau\geq 1$ be the period of $q$ and
$\chi\colon [-1,1]^d\to M$ be some coordinates such that
\begin{itemize}
\item[--] $\chi(0)=q$;
\item[--] the image $D^u:=\chi((-1,1)\times\{0\}^{d-d^u-1}\times (-1,1)^{d^u})$ is contained in $W^u_{loc}(q)$;
\item[--] the image $D^{uu}:=\chi(\{0\}^{d-d^u}\times (-1,1)^{d^u})$ is contained in $W^{uu}_{loc}(q)$;
\item[--] the image $D^{u,+}:=\chi([0,1)\times\{0\}^{d-d^u-1}\times (-1,1)^{d^u})$ is contained in $R^u(q)$;
\item[--] $f^{-\tau}(\overline{D^u})$ is contained in $D^{u}$.
\end{itemize}
One deduces that there exists $n_0\geq 0$ such that:
\begin{itemize}
\item[(i)] Any point $z$ close to
$\overline{D^{u,+}}\setminus f^{-\tau}(D^u)$ has an iterate $f^k(z)$, $|k|\leq n_0$, in $V^u$.
\end{itemize}

The graph transform argument (see for instance~\cite[section 6.2]{KH})
gives the following generalization of the $\lambda$-lemma.

\begin{af}
There exists $N\geq 0$ and, for all $a$ in a neighborhood $A$ of $a_0$,
there exist some decreasing sequences of disks $(D_{a,n})$ of $[-1,1]^{d^u+1}$
and $(D'_{a,n})$ of $\{0\}\times [-1,1]^{d^u}$ which contain $0$ and
such that for any $n\geq N$ one has, in the coordinates of $\chi$:
\begin{itemize}
\item[--] $f^{n\tau}(D_{a,n})$ is the graph of a function $D^u\to \RR^{d-d^u-1}$
that is $C^1$-close to $0$;
\item[--] $f^{n\tau}(D'_{a,n})$ is the graph of a function $D^{uu}\to \RR^{d-d^u}$
that is $C^1$-close to $0$.
\end{itemize}
\end{af}

Let us consider $a\in A$.
The image by $f^{n\tau}$ of each component of $D_{a,n}\setminus D'_{a,n}$ is
contained in a small neighborhood of a component of $D^{u}\setminus D^{uu}$.
The graph $f^{n\tau}(D'_{a,n})$ which is transverse to a constant cone field
around the central direction at $q$.
Since $\varphi$ is coherent with the central orientation at $q$, one deduces that
\begin{itemize}
\item[(ii)] $f^{n\tau}\circ\varphi_a\left(([0,1]\times [-1,1]^{d^u})\cap D_{a,n}\right)$
is contained in a small neighborhood of $D^{u,+}$.
\end{itemize}

For $\delta>0$ small, any point $z\in \varphi_a([-\delta,\delta]\times [-\delta,\delta]^{d^u})$ different from $\varphi_a(0)$
belongs to some $D_{a,n}\setminus D_{a,n+1}$, with $n\geq N$. Consequently:
\begin{itemize}
\item[(iii)] Any $z\in \varphi_a([-\delta,\delta]\times [-\delta,\delta]^{d^u})\setminus
\{\varphi_a(0)\}$ has a forward iterate in $D^u\setminus f^{-1}(D^u)$.
\end{itemize}
Putting the properties (i-iii) together, one deduces the announced property.
\end{proof}

\begin{proof}[Proof of proposition~\ref{propMecanisme}]
We denote by $d^s+1$ (resp. $d^u+1$) the stable dimension of $p$ (resp. the unstable dimension of $q$)
so that the dimension of $M$ satisfies $d=d^s+d^u+1$.
Consider a $C^1$-embeddeding $\varphi: [-1,1]^{d} \to M$ with $\varphi(0)=x$ such that:
\begin{itemize}
\item[--] $\varphi(\{0\}\times [-1,1]^{d^s} \times \{0\}^{d^u})$ is contained in $W^s(q)$;
\item[--] $\varphi(\{0\}\times \{0\}^{d^s} \times [-1,1]^{d^u})$ is contained in $W^u(p)$;
\item[--] $D_0\varphi.(1,0^{d^s},0^{d^u})$ is tangent to $E^c_x$
and has positive orientation.
\end{itemize}

Note that all the restrictions of $\varphi$ to
$[-1,1]\times \{a^s\}\times [-1,1]^{d^u}$
for $a^s\in \RR^{d^s}$ close to $0$,
are coherent with the central orientation at $q$.

Consider a compact set $\Delta^u\subset R^u(q)\backslash \{q\}$
that meets each orbit of $R^u(q)\backslash \{q\}$. Since $\cC$ is closed
and $q$ is unstable cuspidal,
there is a neighborhood $V^u$ of $\Delta^u$ in $M$ that is disjoint from $\cC$.
The lemma~\ref{LemaLambdaLema} can be applied: the points in
$\varphi([0,\delta]\times \{a^s\}\times [-\delta,\delta]^{d^u})$
distinct from $\varphi(0,a^s,0^{d^u})$ have an iterate in $V^u$, hence do not belong to $\cC$.
This shows that
$$\cC\cap \varphi\left([0,\delta]\times [-\delta,\delta]^{d-1}\right)
\subset \varphi\left(\{0\}\times [-\delta,\delta]^{d^s}\times \{0\}^{d^u}\right).$$

From (H2), if one reverses the central orientation and if one considers
the dynamics of $f^{-1}$, then all the restrictions of $\varphi$ to
$[-1,1]\times [-1,1]^{d^s}\times \{a^u\}$
for $a^u\in \RR^{d^u}$ close to $0$,
are coherent with the central orientation at $p$.
One can thus argues analogously and gets:
$$\cC\cap \varphi\left([-\delta,0]\times [-\delta,\delta]^{d-1}\right)
\subset \varphi\left(\{0\}\times \{0\}^{d^s}\times [-\delta,\delta]^{d^u}\right).$$
Both inclusions give that
$$\cC\cap \varphi\left([-\delta,\delta]^{d}\right)=\{\varphi(0)\},$$
which says that $x=\varphi(0)$ is isolated in $\cC$.
\end{proof}

\section{Construction of the example}\label{sectionExample}
In this part we build a collection of diffeomorphisms satisfying the properties (I) and (II) stated in the theorem.
The construction will be made only in dimension $3$ for notational purposes.
The generalization to higher dimensions is straightforward.

\subsection{Construction of a diffeomorphism}
Let us consider an orientation-preserving $C^\infty$ diffeomorphism $H$ of the plane $\RR^2$
and a closed subset $D=D^-\cup C\cup D^+$ such that:
\begin{itemize}
\item[--] $H(\overline D)\subset \interior(D)$ and
$H(\overline{D^-\cup D^+})\subset \interior(D^-)$;
\item[--] the forward orbit of any point in $D^-$ converges towards a sink $S\in D^-$;
\item[--] $C$ is the cube $[0,5]^2$ whose maximal invariant set is a hyperbolic horseshoe.
\end{itemize}
On $C\cap H^{-1}(C)$ the map $H$ is piecewise linear, it preserves and contracts by $1/5$ the horizontal direction
and it preserves and expands by $5$ the vertical direction (see figure~\ref{herradura}):
\begin{itemize}
\item[--] The set $C\cap H(C)$ is the union of $4$ disjoint vertical bands $I_1,I_2,I_3,I_4$ of width $1$.
We will assume that $I_1\cup I_2 \en (0,2+\frac 1 3)\times [0,5]$ and
$I_3\cup I_4 \en (2+\frac 2 3,5)\times [0,5]$.
\item[--] The preimage $H^{-1}(C)\cap C$ is the union of $4$ horizontal bands $H^{-1}(I_i)$.
We will assume that $H^{-1}(I_1\cup I_2) \en [0,5]\times (0,2+\frac 1 3)$ and
$H^{-1}(I_3\cup I_4) \en [0,5]\times (2+\frac 2 3,5)$.
\end{itemize}
\begin{figure}[ht]\begin{center}
\input{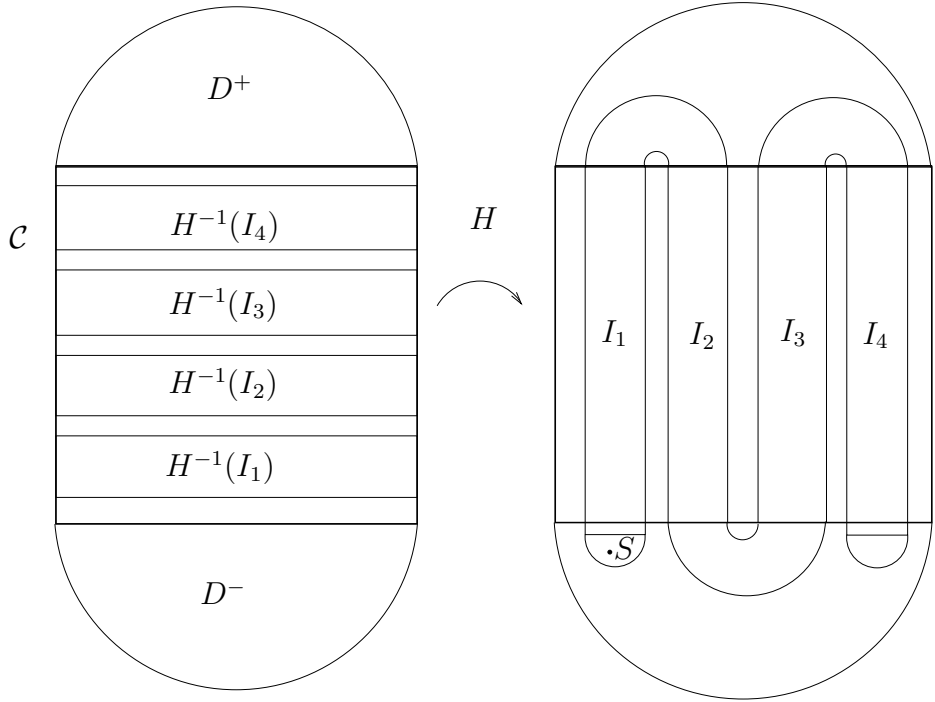}
\caption{\small{The map $H$.}} \label{herradura}
\end{center}\end{figure}

We define a $C^\infty$ diffeomorphism $F$ of $\RR^3$
whose restriction to a neighborhood of $D\times [-1,6]$
it is a skew product of the form
$$F\colon (x,t)\mapsto (H(x), g_x(t)),$$
where the diffeomorphisms $g_x$ are orientation-preserving and satisfy (see figure \ref{dimension1}):

\begin{itemize}
\item[(P1)] $g_x$ does not depend on $x$ in the sets $H^{-1}(I_i)$ for every $i=1,2,3,4$.
\smallskip

\item[(P2)] For every $(x,t)\in D\times [-1,6]$ one has $4/5 < g_x'(t) <6/5$.
\smallskip

\item[(P3)] $g_x$ has exactly two fixed points inside $[-1,6]$,
which are $\{0,4\}$, $\{3,4\}$, $\{1,2\}$ and $\{1,5\}$, when $x$ belongs to
$H^{-1}(I_i)$ for $i$ respectively equal to $1,2,3$ and $4$. All fixed points are hyperbolic, moreover,  
\begin{itemize}
\item[-] $g'_x(t) < 1$ for $t \in [-1, 3+1/2]$ and $x\in
H^{-1}(I_1) \cup H^{-1}(I_2)$.
\item[-] $g'_x(t) > 1$ for $t\in [1+1/2,6]$ and
$x\in  H^{-1}(I_3) \cup H^{-1}(I_4)$.
\end{itemize}
\smallskip

\item[(P4)] For every $(x,t)\in (D^-\cup D^+)\times[-1,6]$ one has $g_x(t)>t$.
\end{itemize}
\begin{figure}[ht]\begin{center}
\input{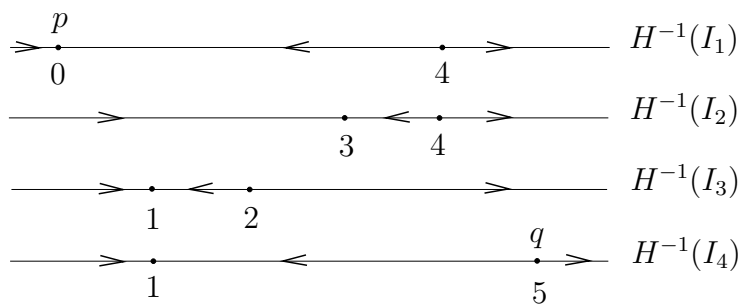}
\caption{\small{The map $g_x$ above each rectangle $H^{-1}(I_i)$.}} \label{dimension1}
\end{center}\end{figure}
We assume furthermore that the following properties are satisfied:
\begin{itemize}
\item[(P5)] $F(D\times [6,8])\subset \interior(D\times [6,8])$;
\item[(P6)] there exists a sink which attracts the orbit of any point of $D\times [6,8]$;
\item[(P7)] $F$ coincides with a linear homothety outside a compact domain;
\item[(P8)] any forward orbit meets $D\times [-1,8]$.
\end{itemize}

One can build a diffeomorphism which coincides with the identity
on a neighborhood of the boundary of $D_0\times (-2,9)$ and coincides with $F$ in $D\times (-1,8)$ ($D_0$ denotes a small neighborhood of $D$ in $\RR^2$).
This implies that, on any $3$-dimensional manifold, every isotopy class of diffeomorphisms contains an element whose restriction to an invariant set is $C^\infty$-conjugated to $F$.
\medskip

On any $3$-dimensional manifold, one can consider an orientation-preserving
Morse-Smale diffeomorphism and by surgery replace the dynamics on
a neighborhood of a sink by the dynamics of $F$. We denote by $f_0$ the obtained diffeomorphism.

\subsection{First robust properties}\label{ss.robust}
We list some properties satisfied by $f_0$, which are also satisfied by any diffeomorphism
$f$ in a small $C^1$-neighborhood $\cU_0$ of $f_0$.
\begin{description}
\item[Fixed points] By (P3),
in each rectangle $\interior(I_i)\times (-1,6)$,
there exists two hyperbolic fixed points $p_i,q_i$. Their stable dimensions are respectively equal to $2$ and $1$.
Since $p_1$ and $q_4$ will play special roles, we shall denote them as $p=p_1$ and $q=q_4$.
\smallskip

\item[Isolation] The two open sets $V_0=\interior(D)\times (-1,8)$
and $V_1= V_0\setminus (C\times [-1,6])$ are isolating blocks, i.e. satisfy $f(\overline{V_0}) \en V_0$ and $f(\overline{V_1}) \en V_1$.

For $V_0$, the property follows immediatly from the construction.
The closure of the second set $V_1$ can be decomposed as the union of:
\begin{itemize}
\item[--] $D^+\times [-1,6]$, which is mapped into $(D^-\times [-1,6])\cup (D\times [6,8])$,
\item[--] $D^-\times [-1,6]$ which is also mapped into $(D^-\times [-1,6])\cup (D\times [6,8])$
and moreover has a foward iterate in $D\times [6,8]$ by (P4),
\item[--] $D\times [6,8]$ which is mapped into itself and whose limit set is a sink.
\end{itemize}

Hence, any chain-recurrence class which meets the rectangle $C \times [-1,6]$
is contained inside. The maximal invariant set of $f$ in $C\times [-1, 6]$
will be denoted by $\cC_f$.

Any chain-recurrence class which meets $V_1$ coincides with the sink of $D\times [6,9]$.
\smallskip

\item[Partial hyperbolicity (property (III) of the theorem)] On $C\times [-1,6]\subset \RR^3$,
there exists some narrow cone fields $\cE^s,\cE^{cs}$
around the coordinate direction $(1,0,0)$ and the plane $(x,0,z)$ which are invariant by
$Df^{-1}$. The vectors tangent to $\cE^s$ are uniformly expanded by $Df^{-1}$.
Similarly there exists some forward invariant cone fields $\cE^u,\cE^{cu}$
close to the direction $(0,1,0)$ and the plane $(0,y,z)$.

In particular $\cC_f$ is partialy hyperbolic.
Moreover the tangent map $Df$ preserves the orientation of the central direction such that
any positive unitary central vector is close to the vector $(0,0,1)$.
\smallskip

\item[Central expansion] Property (P2) holds for $f$ when one replaces
the derivative $g'_x(t)$ by the tangent map  $\|Df|_{E^c}(x,t)\|$ along the central bundle.
\smallskip

\item[Properties (H2) and (H3)] The point $p$ is stable cuspidal and the point $q$ is unstable cuspidal.
More precisely the left half plaque of $W^s(p)$ and the right half plaque of $W^u(q)$ are disjoint from $\cC_f$:
since the chain-recurrence classes of $p$ and $q$ are contained in $\cC_f$ this implies property (H3).
Moreover if there exists an intersection point $x\in W^u(p)\cap W^s(q)$ for $f$, then by the isolating property
it is contained in $\cC_f$. By preservation of the central orientation,
(H2) holds also.

Let us explain how to prove these properties:
it is enough to discuss the case of the left half-plaque of $W^{s}(p)$ and
(arguing as in remark~\ref{RmqCuspidalRobust}) to assume that $f=f_0$.
From (P2) and (P3), we have:
\begin{itemize}
\item[--] every point in $C\times [-1,0)$ has a backward iterate outside $C\times [-1,6]$;
\item[--] the same holds for every point in $(C\setminus I_1)\times \{0\}$;
\item[--] any point in $I_1\times \{0\}$ has some backward image in $(C\setminus I_1)\times \{0\}$, unless it belongs to $W^{u}(p)$.
\end{itemize}
Combining these properties, one deduces that the connected component of
$W^{s}(p)\cap (C\times [-1,0])$ containing $p$ intersects $\cC_f$ only at $p$.
Note that this is a left half plaque of $W^s(p)$, giving the required property.
\smallskip

\item[Hyperbolic regions] By (P3),
the maximal invariant set in $Q_p:=[0,5]\times [0,2+\frac 1 3] \times [-1,3+\frac 1 2]$
and $Q_q:=[0,5]\times [2+\frac 2 3, 5] \times [1+\frac 1 2,6]$ are two locally maximal transitive hyperbolic sets, denoted by $K_p$ and $K_q$.
Their stable dimensions are $2$ and $1$ respectively.
The first one contains $p,p_2$,
the second one contains $q,q_3$.
\smallskip

\item[Tameness (property (4) of the theorem)] Since $f_0$ has been obtained by
sur\-ge\-ry of a Morse-Smale diffeomorphism, the chain-recurrent set
in $M\setminus \cC_f$ is a finite union of hyperbolic periodic orbits.
\end{description}
\medskip

Any $x\in \cC_f$ has a strong stable manifold $W^{ss}(x)$.
Its \emph{local} strong stable manifold $W^{ss}_{loc}(x)$ is the connected component containing $x$ of the intersection $W^{ss}(x) \cap C\times [-1,6]$.
It is a curve bounded by $\{0,5\} \times [0,5] \times [-1,6]$.
Symmetrically, we define $W^{uu}(x)$ and $W^{uu}_{loc}(x)$.
%
%For hyperbolic periodic points $x\in \C$ we also define $W^{s}_{loc}(x)$
%as the connected component of $W^{s}(x)\cap (C\times [-1,6])$
%containing $x$ and $W^u_{loc}(x)$ similarly.

\subsection{Central behaviours of the dynamics}
We analyze the local strong stable and strong unstable manifolds
of points of $\cC_f$ depending on their central position.
\begin{lema}\label{lemaRegions} There exists an open set $\cU_1 \en \cU_0$ such that for every $f\in \cU_1$ and $x\in\cC_f$:
\begin{itemize}
\item[(R1)] If $x\in R_1 := C\times [-1,4+\frac 1 2]$,
then $W^{uu}_{loc}(x) \cap W^s(p) \neq \emptyset$.
\item[(R2)] If $x\in R_2 := C\times [\frac 1 2,6]$,
then $W^{ss}_{loc}(x) \cap W^u(q) \neq \emptyset$.
\item[(R3)] If $x\in R_3 := C\times [\frac 1 2, 2+ \frac 1 2]$,
then $W^{ss}_{loc}(x) \cap W^{uu}_{loc}(y) \neq \emptyset$ for some $y\in K_p$.
\item[(R4)] If $x\in R_4 := C\times [2+\frac 1 2, 4+\frac 1 2]$,
then $W^{uu}_{loc}(x) \cap W^{ss}_{loc}(y) \neq \emptyset$ for some $y\in K_q$.
\end{itemize}
Moreover $p_2$ belongs to $R_2$ and $q_3$ belongs to $R_1$.
\end{lema}
\begin{proof}
Properties (R1) and (R2) follow directly from the continuous variation of the stable and unstable manifolds. Similarly $p_2\in R_2$ and $q_3\in R_1$ by continuity.

We prove (R3) with classical blender arguments (see \cite{persistence} and \cite[chapter 6]{BDV} for more details). The set $K_p$ is a called \emph{blender-horseshoes} in \cite[section 3.2]{BDtang}.

A \emph{cs-strip} $\cS$ is the image by a diffeomorphism $\phi: [-1,1]^2 \to Q_p=[0,5]\times [0,2+\frac 1 3]\times [-1,3+\frac 1 2]$ such that:
\begin{itemize}
\item[--] The surface $\cS$ is tangent to the center-stable cone field and meets $C\times [\frac 1 2, 2+\frac 1 2]$.
\item[--] The curves $\phi(t,[-1,1])$, $t\in [-1,1]$, are tangent to the strong stable cone field and
crosses $Q_p$, i.e. $\phi(t,\{-1,1\}) \subset \{0, 5\}\times [0,2+\frac 1 3] \times [-1,3+\frac 1 2]$.
\item[--] $\cS$ does not intersect $W^{u}_{loc}(p)\cup W^u_{loc}(p_2)$.
\end{itemize}
The \emph{width} of $\cS$ is the minimal length of the curves contained in $\cS$,
tangent to the center cone, and that joins $\phi(-1,[-1,1])$ and $\phi(1,[-1,1])$.

Condition (P2) is important to get the following (see \cite[lemma 6.6]{BDV} for more details):

\begin{af} There exists $\lambda>1$ such that if $\cS$ is a $cs$-strip of width $\eps$, then, either $f^{-1}(\cS)$ intersects $W^{u}_{loc}(p)\cup W^u_{loc}(p_2)$
or it contains at least one $cs$-strip with width $\lambda \eps$.
\end{af}
\begin{proof}
Using (P2), the set $f^{-1}(\cS) \cap C\times [-1,6]$ is the union of two bands crossing $C\times [-1,6]$:
the first has its two first coordinates near $H^{-1}(I_1)$, the second near $H^{-1}(I_2)$.
Their width is larger than $\lambda \eps$ where $\lambda>1$ is a lower bound of the expansion of $Df^{-1}$
in the central direction inside $Q$. We assume by contradiction that none of them intersects
$W^{u}_{loc}(p)\cup W^{u}_{loc}(p_2)$, nor $C\times [\frac 1 2, 2+\frac 1 2]$.

Since $\cS$ intersects $C\times[\frac 1 2, 2+\frac 1 2]$, from conditions (P2) and (P3)
the first band intersects $C\times[\frac 1 2, 4]$. By our assumption
it is thus contained in $C\times(2+\frac 1 2, 4]$. Using (P2) and (P3) again,
this shows that $\cS$ is contained in $C\times (2, 4]$.
The same argument with the second band shows that $\cS$ is contained
in $C\times [-1, 2)$, a contradiction.
\end{proof}

Repeating this procedure, we get an intersection point between $W^{u}_{loc}(p)\cup W^u_{loc}(p_2)$
and a backward iterate of the $cs-$strip. It gives in turn a transverse intersection point $z$
between the initial $cs-$strip and $W^{u}(p)\cup W^u(p_2)$. By construction, all the past iterates of $z$ belong to $Q_p$. Hence $z$ has a well defined local strong unstable manifold. In particular, the intersection $y$ between $W^{uu}_{loc}(z)$ and $W^s_{loc}(p)$ (which exists by (R1)) remains in $Q_p$
both for future and past iterates, thus, it belongs to $K_p$.

For any point  $x\in \cC_f \cap R_3$, one builds a $cs$-strip by thickening in the central direction the local strong stable manifold. We have proved that this $cs-$strip intersects $W^{uu}_{loc}(y)$ for some $y\in K_p$.
One con consider a sequence of thiner strips. Since $K_p$ is closed and the local strong unstable manifolds vary continuously, we get at the limit an intersection between $W^{ss}_{loc}(x)$ and
$W^{uu}_{loc}(y')$ for some $y'\in K_p$ as desired.

This gives (R3). Property (R4) can be obtained similarly.
\end{proof}
\bigskip

We have controled the local strong unstable manifold of points in $R_1\cup R_4$
and the local strong stable manifold of points in $R_2\cup R_3$.
Since neither $R_1 \cup R_4$ nor $R_2 \cup R_3$ cover completely $C\times [-1,6]$
we shall also make use of the following result:

\begin{lema}\label{lemaUniquePoints} For every diffeomorphism
in a small $C^1$-neighborhood $\cU_2 \en \cU_0$ of $f_0$,
the only point whose complete orbit is contained in $C\times [-1,\frac 1 2]$ is $p$;
symmetrically, the only point whose complete orbit is contained
in $C\times [4+\frac 1 2, 6]$ is $q$.
\end{lema}
\begin{proof}
We argue as for property (H3) in section~\ref{ss.robust}:
the set of points whose past iterates stay in $C\times [-1,\frac 1 2]$ is the local strong unstable manifold of $p$. Since $p$ is the only point in its local unstable manifold
whose future iterates stay in $C\times [-1,\frac 1 2 ]$ is $p$ we conclude.
\end{proof}

\subsection{Properties (I) and (II) of the theorem}
We now check that (I) and (II) hold for the region
$U=\interior(C\times [-1,6])$ and the neighborhood
$\cU:=\cU_1\cap \cU_2$.

\begin{prop}\label{p.approx-strong}
For any $f\in \cU$, $x\in \cC_f$, there are
arbitrarily large $n_q,n_p\geq 0$
such that $W^{uu}_{loc}(f^{n_q}(x))\cap W^{ss}(y_q)\neq \emptyset$ and
$W^{ss}_{loc}(f^{-n_p}(x))\cap W^{uu}(y_p)\neq \emptyset$
for some $y_q\in K_q$, $y_p\in K_p$.
\end{prop}
\begin{proof}
If $\{f^n(x), n\geq n_0\}\subset C\times [4+\frac 1 2, 6]$, for some $n_0\geq 0$,
then $x\in W^{ss}(q)$ by lemma~\ref{lemaUniquePoints}.

In the remaining case, there exist some arbitrarily large forward iterates
$f^n(x)$ in $R_1$, so that $W^{uu}_{loc}(f^n(x))$ meets $W^s(p)$
by lemma~\ref{lemaRegions}.
Since $p$ is homoclinically related with $p_2$, by the $\lambda$-lemma
there exists $k\geq 0$ such that $f^k(W^{uu}_{loc}(f^n(x)))$
contains $W^{uu}_{loc}(x')$ for some $x'\in W^{s}(p_2)\cap R_4$ because $p_2\in R_4$.
By lemma~\ref{lemaRegions}, $f^k(W^{uu}_{loc}(f^n(x)))$ intersects $W^{ss}_{loc}(y'_q)$
for some $y'_q\in K_q$ showing that
$W^{uu}_{loc}(f^n(x))\cap W^{ss}(y_q)\neq \emptyset$ with $y_q=f^{-k}(y'_q)$
in $K_q$.

We have obtained the first property in all the cases. The second property is similar.
\end{proof}
\medskip

The following corollary (together with the isolation property of section~\ref{ss.robust})
implies that for every $f\in \cU$, the properties (I) and (H1) are verified.

\begin{cor}\label{PropH1} For every $f\in \cU$ the set $\cC_f$ is
contained in a chain-transitive class.
\end{cor}
\begin{proof}
For any $\varepsilon>0$ and $x\in \cC_f$, there exists a
$\varepsilon$-pseudo-orbit $p=x_0,x_1,\dots,x_n=p$, $n\geq 1$, which contains $x$.
Indeed by proposition~\ref{p.approx-strong},
and using that $K_p,K_q$ are transitive and contain
respectively $p$ and $q_3$, there exists a $\varepsilon$-pseudo-orbit
from $p$ to $q_3$ which contains $x$.
By lemma~\ref{lemaRegions}, the unstable manifold of $q_3$ intersects the stable manifold of $p$, hence there
exists a $\varepsilon$-pseudo-orbit from $q_3$ to $p$.
We take the concatenation of these pseudo-orbits.
\end{proof}
\bigskip

Now, we show that (H2) holds for a $C^r$ dense set $\cD$ of $\cU$.
Since (H1) and (H3) are satisfied, proposition~\ref{propMecanisme}
implies that the property (II) of the theorem holds with the set $\cD\subset \cU$.
In fact, as we noticed in section~\ref{ss.robust} it is enough to get the following.

\begin{cor}\label{corH2densamente} For every $r\geq 1$,
the set $$\cD=\{f\in \cU, \; W^{u}(p)\cap W^{s}(q)\neq \emptyset\}$$
is dense in $\cU\cap \Diff^r(M)$. It is a countable union of one-codimensional submanifolds.
\end{cor}
In the $C^1$ topology, this result is direct consequence of the connecting lemma
(together with proposition~\ref{p.approx-strong}).
The additional structure of our specific example allows to make these perturbations in any $C^r$-topology.
\begin{proof} Fix any $f\in \cU$.
By proposition~\ref{p.approx-strong}, there exists $x\in K_q$ such that $W^{u}(p)$
intersects $W^{ss}(x)$ at a point $y$ (notice that $y\not \in K_q\cup \{p\}$).
Let $U$ be a neighborhood of $y$ such that:

\begin{itemize}
\item[--] $U$ is disjoint from the iterates of $y$, i.e.
$\{f^n(y) \ : \ n \in \ZZ \} \cap U = \{y\}$;
\item[--] $U$ is disjoint from $K_q\cup \{p\}$.
\end{itemize}

Given a $C^r$ neighborhood $\cV$ of the identity, there exists
a neighborhood $V\en U$ of $y$ such that,
for every $z\in V$, the set $\cV$ contains a diffeomorphism $g_z$ which coincides
with the identity in the complement of $U$ and maps $y$ at $z$.

Since $K_q$ is locally maximal, there exists $\bar x\in K_q\cap W^s(q)$ near $x$.
In particular $W^{ss}_{loc}(\overline{x})$ intersects $V$ in a point $z$
whose backward orbit is disjoint from $U$.

For the diffeomorphism $h= g_z\circ f$ (which is $C^r$-close to $f$)
the manifolds $W^{s}(q)$ and $W^{u}(p)$ intersect.
Indeed both $f$ and $h$ satisfy $f^{-1}(y)\in W^{u}(p)$
and $z\in W^{ss}_{loc}(\overline{x})$.
Since  $W^{ss}_{loc}(\overline{x})\en W^{ss}(q)$
and $h(f^{-1}(y))= z$ we get the conclusion.

For each integer $n\geq 1$, the manifolds
$f^n(W^{uu}_{loc}(p))$ and $W^{ss}_{loc}(q)$ have disjoint boundary
and intersect in at most finitely many points.
One deduces that the set $\cD_n$ of diffeomorphisms such that they intersect
is a finite union of one-codimensional submanifold of $\cU$.
The set $\cD$ is the countable union of the $\cD_n$.
\end{proof}

\subsection{Other properties}\label{sectionAutresProprietes}
We here show properties (1), (2) and (3) of the theorem.

\begin{prop}\label{propPerteneceHomoclinica} For every $f\in \cU$ and $x\in \cC_f$ we have:
\begin{itemize}
\item[--] If $x\not\in W^{s}(q)$, there exist large $n\geq 0$ such that
$W^{uu}_{loc}(f^n(x))\cap W^s(p)\neq \emptyset$.
\item[--] If $x \not\in W^{u}(p)$, there exists large $n\geq 0$ such that
$W^{ss}_{loc}(f^{-n}(x))\cap W^u(q)\neq\emptyset$.
\end{itemize}
Moreover, in the first case $x$ belongs to the homoclinic class of $p$ and in the second
it belongs to the homoclinic class of $q$.
\end{prop}

\begin{proof}
By lemma~\ref{lemaUniquePoints}, any point $x\in \cC_f\setminus W^{s}(q)$ has arbitrarily large iterates
$f^n(x)$ in $R_1$, proving that $W^{uu}_{loc}(f^n(x))\cap W^s(p)\neq \emptyset$.

In particular, $W^s(p)$ intersects transversaly $W^{uu}_{loc}(x)$
at points arbitrarily close to $x$. On the other hand by proposition~\ref{p.approx-strong},
there exists a sequence $z_n$ converging to $x$ and
points $y_n\in K_p$ such that $z_n\in W^{u}(y_n)$ for each $n$, proving that
$W^{uu}_{loc}(z_n)$ intersects $W^u(p)$ transversaly at a point close to $x$ when
$n$ is large.
By the $\lambda$-lemma, $W^{uu}_{loc}(y_n)$ is the $C^1$-limit of a sequence of
discs contained in $W^u(p)$. This proves that $W^u(p)$ and $W^s(p)$
have a transverse intersection point close to $x$, hence $x$ belongs to the
homoclinic class of $p$.

The other properties are obtained analogously.
\end{proof}
\medskip

Let $H_f$ denotes the homoclinic class of $p$.
The next gives property (1) of the theorem.

\begin{cor}\label{corHomoclinicCoincide} For every $f\in \cU$,
the homoclinic class of any hyperbolic periodic point of $\cC_f$ coincides with $H_f$. Moreover, the periodic points in $\cC_f$ of the same stable index are homoclinically related.
\end{cor}
\begin{proof}

Let $z\in \cC_f$ be a hyperbolic periodic point whose stable index is $2$.
By proposition~\ref{p.approx-strong} $W^{ss}(z)$ intersects $W^{uu}_{loc}(y)$ for some $y\in K_p$, this implies that $W^s(z)$ intersects $W^{uu}_{loc}(y)$ and since $W^{uu}_{loc}(y)$ is accumulated by $W^{u}(p)$ we get that $W^s(z)$ intersects $W^u(p)$. Now, by proposition~\ref{propPerteneceHomoclinica}, $W^{u}(z)$ intersects $W^{s}(p)$. Moreover the partial hyperbolicity
implies that the intersections are transversal, proving that
$z$ and $p$ are homoclinically related.
One shows in the same way that any hyperbolic periodic point whose stable index
is $1$ is homoclinically related to $q$.

It remains to prove that the homoclinic classes of $p$ and $q$ coincide.
The homoclinic class of $q$ contains a dense set of points $x$ that are homoclinic
to $q_3$. In particular, $x$ does not belong to $W^{u}(q)$, hence belongs to the homoclinic
class of $p$ by proposition~\ref{propPerteneceHomoclinica}.
This gives one inclusion. The other one is similar.
\end{proof}
\medskip

Properties (2) and (3) of the theorem follow from corollary~\ref{corH2densamente}
and the following.

\begin{cor}\label{PropUnicasIntersecciones} For every $f\in \cU$
we have $\cC_f \backslash H_f= W^{s}(q)\cap W^{u}(p)$.
\end{cor}
\begin{proof}
By corollary~\ref{corHomoclinicCoincide}, a point $x\in \cC_f \backslash H_f$
does not belong to the homoclinic class of $q$ (nor to the homoclinic class of $p$
by definition of $H_f$).
Proposition~\ref{propPerteneceHomoclinica} gives
$\cC_f \backslash H_f\subset W^{s}(q)\cap W^{u}(p)$.
Proposition~\ref{propMecanisme} proves that the points of $W^{s}(q)\cap W^{u}(p)$
are isolated in $\cC_f$. Since any point in a non-trivial homoclinic class
is limit of a sequence of distinct periodic points of the class
we conclude that $W^{s}(q)\cap W^{u}(p)$ and $H_f$ are disjoint.
\end{proof}

The proof of the theorem is now complete.

\end{document}